\documentclass[12pt]{article}

\usepackage{amsmath}
\usepackage{amsthm}
\usepackage{amssymb}
\usepackage{mathptmx}
\usepackage{bbding}
\usepackage{graphicx}
\usepackage{multirow,makecell}

\title{\bf Characterizing isolated singularities of conformal hyperbolic metrics}
\author{Yu Feng$^1$, Yiqian Shi$^2$ and Bin Xu$^3$}
\usepackage{fancyhdr}
\def\nd{\noindent}
\newtheorem{thm}{Theorem}[section]
\newtheorem{lem}{Lemma}[section]

\newtheorem{defi}{Definition}[section]
\newtheorem{rem}{Remark}[section]

\newtheorem{prob}{Problem}[section]

\begin{document}

\maketitle

\nd{\small  $^{1,2,3}$Wu Wen-Tsun Key Laboratory of Mathematics, USTC, Chinese Academy of Sciences\\
$^{1,2,3}$School of Mathematical Sciences, University of Science and Technology of China.\\
No. 96 Jinzhai Road, Hefei, Anhui Province\  230026\  P. R. China.}

\nd {\small $^{1}$\Envelope yuf@mail.ustc.edu.cn \quad  $^{2}$ yqshi@ustc.edu.cn\quad  $^3$ bxu@ustc.edu.cn}
\par\vskip0.5cm

\nd {\small {\bf Abstract:}
We find the explicit local models of isolated singularities of conformal hyperbolic metrics by Complex Analysis, which is interesting in its own and could potentially be extended to high-dimensional case.\\

\nd{\bf Keywords.} conformal hyperbolic metrics, conical singularity, cusp singularity, developing map\\

\nd {\bf 2010 Mathematics Subject Classification.} Primary 51M10; Secondary 34M35

\section{Introduction}
\paragraph{}  Let $\Sigma$ be a Riemann surface and $\textup{D}=\sum\limits_{i=1}^\infty(\theta_{i}-1)p_{i}$ a ${\Bbb R}$-divisor on $\Sigma$ such that $0\leq \theta_i\not=1$, where $\{p_{i}\}_{i=1}^\infty\subset \Sigma$ is a closed discrete subset. We denote by $M(\Sigma)$ the set of $C^{\infty}$ conformal metrics of constant curvature $-1$ on a Riemann surface $\Sigma$. We call $\mathrm{d} \sigma^2$ a (singular) conformal hyperbolic metric {\it representing $D$}  if and only if
\begin{itemize}

              \item   $\mathrm{d} \sigma^2\in M(\Sigma \backslash {\rm supp}\, D)$, where ${\rm supp}\, D= \{p_1,\cdots, p_n,\cdots\}$.

                 \item \ If $\theta_i>0$, then $\mathrm{d} \sigma^2$\ {\it has a conical singularity at $p_{i}$ with cone angle\ $2\pi\theta_{i}>0$}. That is, in a neighborhood $U$ of  $p_{i}$, $\mathrm{d} \sigma^2=e^{2u}{\vert \mathrm{d} z \vert}^{2}$, where $z$ is a complex coordinate of $U$  with\ $z(p_{i})=0$ and\ $u-(\theta_i-1)\ln \vert z \vert$\ extends to a continuous function in $U$.

         \item \ If $\theta_i=0$, then $\mathrm{d} \sigma^2$\ {\it has a cusp singularity at\ $p_{i}$}. That is, in a neighborhood $V$ of\ $p_{i}$,\ $\mathrm{d} \sigma^2=e^{2u}{\vert \mathrm{d} z \vert}^{2}$, where\ $z$\ is a complex coordinate of $V$ with\ $z(p_{i})=0$ and\ $u+\ln\, \vert z \vert+\ln\,(-\ln\, \vert z \vert)$\ extends to a continuous function in $V$.

 \end{itemize}

There have been some studies on the local behavior of a conformal hyperbolic metric near an isolated singularity. Nitsche \cite{N57}, Heins \cite{HE62}, Chou and Wan \cite{CW94,CW95} proved that an isolated singularity of a conformal hyperbolic metric must be either a conical singularity or a cusp one. They all obtained the result by studying the behaviour of the solutions of the Liouville equation
$$\Delta u =e^{2u}$$
near isolated singularities. Yamada \cite{Ya88} considered the same problem from the perspective of complex analysis. However, all of them  only gave an asymptotic model for a hyperbolic metric near an isolated singularity. We want to seek for an explicit local model.

  For this purpose, we did some explorations. Firstly, in \cite{LLX17}, by using PDEs Bo Li, Long Li, the first and the third authors proved the following lemma:

\begin{lem}
\label{lem:Schwarzian}
Let\ $\mathrm{d} \sigma^2$\ be a conformal hyperbolic metric on a Riemann surface\ $\Sigma$, and suppose\ $\mathrm{d} \sigma^2$\ represents a divisor\ $\textup{D}=\sum\limits_{i=1}^\infty(\theta_{i}-1)p_{i},\  0\leq \theta_{i}\not=1$.
Suppose that\ $F:\Sigma\backslash {\rm supp}\, D\longrightarrow \mathbb{D}$\ is a developing map of\ $\mathrm{d} \sigma^2$ (see Defiition \ref{defi:developing map}). Then the Schwarzian derivative\ $\{F,z\}=\frac{F^{'''}(z)}{F^{'}(z)}-\frac{3}{2}\Big(\frac{F^{''}(z)}{F^{'}(z)}\Big)^{2}$ of\ $F$\ equals
\[
\{F,z\}=\frac{1-\theta_{i}^{2}}{2z^{2}}+\frac{d_{i}}{z}+\phi_{i}(z)
\]
in a neighborhood\ $U_{i}$\ of\ $p_{i}$\ with complex coordinate\ $z$ and\ $z(p_{i})=0$, where \ $d_{i}$\ are constants and \ $\phi_{i}$\ are holomorphic functions in\ $U_{i}$, depending on the complex coordinate\ $z$.
\end{lem}

\noindent Based on Lemma \ref{lem:Schwarzian}, we obtained the local expressions of developing maps near isolated singularities of hyperbolic metrics (see \rm{\cite[Lemma 2.4]{FSX}}). In this process, we solved a Fuchsian equation of second order. Using these expressions, we finally got the local model of a singular conformal hyperbolic metric. Therefore, this proof process has some twists and turns. After these explorations and reading \cite{Ya88}, we speculated that there should be a direct proof by using Complex Analysis only, which motivated this manuscript. In this note, we complete this modest project.

Below we present the main result of this note.

\begin{thm}\rm{\cite[Theorem 1.2]{FSX}}
\label{thm:main}
\emph{
Let\ $\mathrm{d} \sigma^2$\ be a conformal hyperbolic metric on the punctured disk\ $\mathbb{D}^{*}=\{{w\in\mathbb{C}|0<\vert w \vert<1}\}$. Then $0$ is either a conical singularity or a cusp singularity of $\mathrm{d} \sigma^2$.  If\ $\mathrm{d} \sigma^2$\ has a conical singularity at\ $w=0$\ with the angle\ $2\pi\alpha>0$, then there exists a complex coordinate\ $z$ on\ $\Delta_{\varepsilon}=\{{w\in\mathbb{C}|\vert w \vert<\varepsilon}\}$\ for some\ $\varepsilon>0$ with\ $z(0)=0$ such that
$$ \mathrm{d} \sigma^2|_{\Delta_{\varepsilon}}=\frac{4\alpha^2\vert z \vert^{2\alpha-2}}{\big(1-\vert z \vert ^{2\alpha}\big)^2}\vert \mathrm{d} z \vert^2.$$
Moreover,\ $z$ is unique up to replacement by\ $\lambda z$ where\ $\vert \lambda \vert=1$. If\ $\mathrm{d} \sigma^2$\ has a cusp singularity at\ $w=0$, then there exists a complex coordinate\ $z$ on\ $\Delta_{\varepsilon}=\{{w\in\mathbb{C}|\vert w \vert<\varepsilon}\}$\ for some\ $\varepsilon>0$ with\ $z(0)=0$ such that
$$ \mathrm{d} \sigma^2|_{\Delta_{\varepsilon}}=\vert z \vert ^{-2}\big(\ln|z|\big)^{-2}|dz|^{2}.$$
Moreover,\ $z$ is unique up to replacement by\ $\lambda z$ where\ $\vert \lambda \vert=1$.}\\
\end{thm}

From the proof of Theorem \ref{thm:main}, we can directly obtain the local expressions of developing maps near isolated singularities of the metrics.

\begin{thm}
\label{thm:local model of developing map}
  Let\ $F:\Sigma\backslash {\rm supp}\, D\longrightarrow \mathbb{D}$ be a developing map of the singular conformal hyperbolic metric $\mathrm{d} \sigma^2$ {\it representing $D$}. If $p$ is a conical singularity of $\mathrm{d} \sigma^2$, then there exists a neighborhood\ $U$ of\ $p$\ with complex coordinate\ $z$\ and\ $\mathfrak{L}\in \textup{PSU}(1,1)$ such that\ $z(p)=0$\ and\ $G=\mathfrak{L}\circ F$\ has the form\ $G(z)=z^{\alpha}$. If $q$ is a cusp singularity of $\mathrm{d} \sigma^2$, we assume $F:\Sigma\backslash {\rm supp}\, D\longrightarrow \mathbb{H}$ for convenience, where $\mathbb{H}$ is the upper half-plane model of the hyperbolic plane, then there exists a neighborhood\ $V$ of\ $q$\ with complex coordinate\ $z$\ and\ $\mathfrak{L}\in \textup{PSL}(2,\mathbb{R})$, such that\ $z(q)=0$\ and\ $G=\mathfrak{L}\circ F$\ has the form\ $G(z)=-\sqrt{-1}\log z$.
\end{thm}

Moreover, the proof process of Theorem \ref{thm:main} provide a possible approach for studying the codimension-one singularities of complex hyperbolic metrics in higher dimension. We shall investigate the following problem in the near future.
\begin{prob}
\label{pro:local}
What is the asymptotic behavior of a complex hyperbolic metric on $\underbrace{\mathbb{D}\times\cdots\times\mathbb{D}}\limits_n\backslash \{z_{1}z_{2}\cdots z_{n}=0\}$ for $n\geq2$?
\end{prob}

We organize the left part of the manuscript as follows. In section 2, we at first give some knowledge of hyperbolic geometry that we need to use. Then we state the definition and properties of developing maps. Section 3 is the proof for Theorem \ref{thm:main} and Theorem \ref{thm:local model of developing map}.

\section{Preliminaries}

\subsection{Conformal isometries of the hyperbolic plane}
We will work with both the Poincar{\' e} disk model\ $\left(\mathbb{D}=\{z\in \mathbb{C}:\vert z \vert<1\},\ \mathrm{d} \sigma^2_{\mathbb{D}}=\frac{4\vert dz \vert^{2}}{(1-\vert z \vert^{2})^{2}}\right)$ and the upper half-plane model \ $\left(\mathbb{H}=\{z\in \mathbb{C}:\textup{Im z}>0\},\ \mathrm{d} \sigma^2_{\mathbb{H}}=\frac{\vert dz \vert^{2}}{(\textup{Im z})^{2}}\right)$ of the hyperbolic plane at convenience.
We denote
$$\textup{PSU(1,1)}=\{z\longmapsto\frac{az+b}{\overline{b}z+\overline{a}}:\ a,\ b\in \mathbb{C},\ \vert a \vert^2-\vert b \vert^2=1\}$$
and
$$\textup{PSL}(2,\mathbb{R})=\{z\longmapsto\frac{az+b}{cz+d}:\ a,\ b,\ c,\ d\in \mathbb{R},\ ad-bc=1\}$$
the group of all orientation-preserving isometries of\ $\mathbb{D}$ and\ $\mathbb{H}$, respectively.

\begin{defi}\rm{\cite[p.136]{R06}}.
   If\ $\mathfrak{L},\ \mathfrak{K}\in\textup{I}(\mathbb{D})$ and\ $\mathfrak{L}$ is not the identity map of ${\Bbb D}$, then\ $\mathfrak{L}$ has a fixed point in\ $\overline{\mathbb{D}}$, where $\textup{I}(\mathbb{D})$ is the isometry group of $\mathbb{D}$.. The transformation\ $\mathfrak{L}$ is said to be

\begin{itemize}
         \item {\it elliptic} if\ $\mathfrak{L}$ fixes a point of\ $\mathbb{D}$;

         \item {\it parabolic} if\ $\mathfrak{L}$ fixes no point of\ $\mathbb{D}$ and fixes a unique point of \ $\partial{\mathbb{D}}$;

         \item {\it hyperbolic} if\ $\mathfrak{L}$ fixes no point of\ $\mathbb{D}$ and fixes two points of \ $\partial{\mathbb{D}}$.

\end{itemize}

\end{defi}

\begin{lem}\rm{\cite[Theorem 2.3, p.67]{A05}}, \rm{\cite[p.172-173]{B83}}.
\label{lem:classification}
\it{In the Poincar{\' e} disk model\ $\mathbb{D}$,  if\ $\mathfrak{L}\in \textup{PSU(1,1)}$ is elliptic, then there exists\ $\mathfrak{K}\in \textup{PSU(1,1)}$ such that\ $\mathfrak{K}\circ\mathfrak{L}\circ\mathfrak{K^{-1}}(z)=e^{i\theta}z$\ for some real number\ $\theta$.

In the upper half-plane model\ $\mathbb{H}$,  if\ $\mathfrak{L}\in \textup{PSL}(2,\mathbb{R})$ is parabolic, then there exists\ $\mathfrak{K}\in \textup{PSL}(2,\mathbb{R})$ such that\ $\mathfrak{K}\circ\mathfrak{L}\circ\mathfrak{K^{-1}}(z)=z+t$  for some real number\ $t$.

In the upper half-plane model\ $\mathbb{H}$,  if\ $\mathfrak{L}\in \textup{PSL}(2,\mathbb{R})$ is hyperbolic, then there exists\ $\mathfrak{K}\in \textup{PSL}(2,\mathbb{R})$ such that\ $\mathfrak{K}\circ\mathfrak{L}\circ\mathfrak{K^{-1}}(z)=\lambda z$ for some positive real number\ $\lambda$.}

\end{lem}

\begin{lem}\rm{\cite[Theorem 4.16, p.141]{A05}}

\label{lem:nonincreasing}
\it{
Let $f:\mathbb{H}\rightarrow\mathbb{H}$ be holomorphic. If $f$ is a homeomorphism, then $f\in\textup{PSL}(2,\mathbb{R})$. If $f$ is not a homeomorphism, then $d_{\mathbb{H}}(f(z_{1}),f(z_{2}))<d_{\mathbb{H}}(z_{1},z_{2})$ for all distinct $z_{1},z_{2}\in\mathbb{H}$, where $d_{\mathbb{H}}$ denotes the hyperbolic distance of $\mathbb{H}$.}

\end{lem}

\begin{lem}\rm{\cite[p.176]{B83}}
\label{lem:infimum}
\it{
For any isometry $\mathfrak{L}$ let $m$ be the infimum of $d_{\mathbb{H}}(z,\mathfrak{L}{z})$ taken with respect to $z\in\mathbb{H}$. Then $\mathfrak{L}$ is hyperbolic if and only if $m>0$; if $m=0$, then $\mathfrak{L}$ is elliptic when $m$ is attained and parabolic when $m$ is not attained.}
\end{lem}

\subsection{Developing map}
\paragraph{}A multi-valued locally univalent meromorphic function\ $F$\ on a Riemann surface\ $\Sigma$ is said to be {\it projective} if any two function elements\ $\mathfrak{F_{1}}, \mathfrak{F_{2}}$\ of\ $F$ near a point\ $p\in \Sigma$  are related by a fractional linear transformation\ $\mathfrak{L}\in \textup{PGL}(2,\mathbb{C})$, i.e.,\ $\mathfrak{F_{1}}=\mathfrak{L}\circ \mathfrak{F_{2}}$.

\begin{defi}\label{defi:developing map}
\textup{Let\ $\mathrm{d} \sigma^2$\ be a conformal hyperbolic metric on a Riemann surface\ $\Sigma$, not necessarily compact, representing the divisor\ $\textup{D}$. We call a projective function\ $F:\Sigma\backslash {\rm supp}\, D\longrightarrow \mathbb{D}$\ a {\it developing map} of the metric\ $\mathrm{d} \sigma^2$\ if\ $\mathrm{d} \sigma^2=F^{*}\mathrm{d} \sigma^2_{\mathbb{D}}$, where\ $\mathrm{d} \sigma^2_{\mathbb{D}}=\frac{|\mathrm{d} z|^{2}}{(1-|z|^2)^{2}}$\ is the hyperbolic metric on the unit disc\ $\mathbb{D}$. And $F$ can also be viewed as a locally schlicht holomorphic function from $\widetilde{\Sigma}$ to $\mathbb{D}$, where $\widetilde{\Sigma}$ is the universal cover of $\Sigma\backslash {\rm supp}\, D$.}
\end{defi}

\begin{lem}\rm{\cite[lemma 2.1 and Lemma 2.2]{FSX}}
\label{lem:devloping map}

{\it Let\ $\mathrm{d} \sigma^2$\ be a conformal hyperbolic metric on a Riemann surface\ $\Sigma$, representing the divisor $\textup{D}$. Then there exists a developing map\ $F$\ from\ $\Sigma\backslash {\rm supp}\, D$\ to the unit disc\ $\mathbb{D}$ such that the monodromy of\ $F$\ belongs to \textup{PSU(1,1)} and
\[
\mathrm{d} \sigma^2=F^{*}\mathrm{d} \sigma^2_{\mathbb{D}},
\]
where\ $\mathrm{d} \sigma^2_{\mathbb{D}}=\frac{|\mathrm{d} z|^{2}}{(1-|z|^2)^{2}}$ is the hyperbolic metric on\ $\mathbb{D}$. Moreover, any two developing maps\ $F_{1}$,\ $F_{2}$ of the metric\ $\mathrm{d} \sigma^2$\ are related by a fractional linear transformation\ $\mathfrak{L}\in \textup{PSU}(1,1)$, i.e.,\ $F_{2}=\mathfrak{L}\circ F_{1}$.}
\end{lem}

\begin{rem}
\label{rem:devloping map}
 There exists an analogue of the above lemma on the upper half-plane model $\mathbb{H}$.
\end{rem}

\section{Proof for Theorem \ref{thm:main} and Theorem \ref{thm:local model of developing map}}
\paragraph{} Let $\mathrm{d} \sigma^2\in M(\mathbb{D}^{*})$ and $\mathbb{H}$ be the upper half-plane. Consider the universal covering from $\mathbb{H}$ to $\mathbb{D}^{*}$, $z\mapsto e^{iz}$, whose covering group $\Gamma$ is generated by $\tau(z)=z+2\pi$. Since $(e^{iz})^{*}\mathrm{d} \sigma^2\in M(\mathbb{H})$, there exists a locally schlicht holomorphic function $f$ from $\mathbb{H}$ to $\mathbb{D}$ such that $(e^{iz})^{*}\mathrm{d} \sigma^2=f^{*}\mathrm{d} \sigma^2_{\mathbb{D}}$ by Lemma \ref{lem:devloping map}. Moreover, we obtain the monodromy homomorphism $\mathcal{M}:\Gamma\rightarrow \textup{PSU(1,1)}$. So we have $f\circ\tau=\mathcal{M}(\tau)\circ f$, set $\mathfrak{L}=\mathcal{M}(\tau)$.

\begin{lem}\rm{\cite[Lemma 7]{Ya88}}
\label{lem:not hyperbolic}
\it{
$\mathfrak{L}$ is not a hyperbolic transformation.}
\end{lem}

\begin{proof}
By Lemma \ref{lem:nonincreasing},
\[
d_{\mathbb{D}}(f(z),\mathfrak{L}\circ {f(z)})=d_{\mathbb{D}}(f(z),f(z+2\pi))\leq d_{\mathbb{H}}(z,z+2\pi).
\]
Let $z=iy,\ \gamma(t)=t+iy,\ 0\leq t\leq2\pi$. Then the length of $\gamma(t)$ equals
\[
l(\gamma(t))=\int^{2\pi}_{0}\frac{1}{y}dt=\frac{2\pi}{y}.
\]
So $d_{\mathbb{H}}(z,z+2\pi)\rightarrow0$ as $y\rightarrow +\infty$. Hence $m=\inf d_{\mathbb{D}}(z,\mathfrak{L}\circ {z})=0$, $\mathfrak{L}$ is not hyperbolic by Lemma \ref{lem:infimum}.

\end{proof}

\begin{lem}
\label{lem:main}
The following expressions hold near the origin.

(1) If $\mathfrak{L}$ is parabolic, then

         $$ \mathrm{d} \sigma^2|_{\Delta_{\varepsilon}}=\vert \xi \vert ^{-2}\big(\ln|\xi|\big)^{-2}|d\xi|^{2},$$
where\ $\Delta_{\varepsilon}=\{{w\in\mathbb{C}|\vert w \vert<\varepsilon}\}$\ for some\ $\varepsilon>0$. Moreover,\ $\xi$ is unique up to replacement by\ $\lambda \xi$ where\ $\vert \lambda \vert=1$.

(2) \ If $\mathfrak{L}$ is elliptic, then

       $$ \mathrm{d} \sigma^2|_{\Delta_{\varepsilon}}=\frac{4(k+\alpha)^{2}\vert \xi \vert^{2k+2\alpha-2}}{(1-\vert \xi \vert ^{2k+2\alpha})^{2}}{\vert \mathrm{d} \xi \vert}^{2}.$$
where $0<\alpha<1$, $k$ is a nonnegative integer and\ $\Delta_{\varepsilon}=\{{w\in\mathbb{C}|\vert w \vert<\varepsilon}\}$\ for some\ $\varepsilon>0$. Moreover,\ $\xi$ is unique up to replacement by\ $\lambda \xi$ where\ $\vert \lambda \vert=1$.

(3) \ If $\mathfrak{L}$ is the identity, then
       $$ \mathrm{d} \sigma^2|_{\Delta_{\varepsilon}}=\frac{4k^{2}\vert \xi \vert^{2k-2}}{(1-\vert \xi \vert ^{2k})^{2}}{\vert \mathrm{d} \xi \vert}^{2}.$$
where $k$ is a positive integer and\ $\Delta_{\varepsilon}=\{{w\in\mathbb{C}|\vert w \vert<\varepsilon}\}$\ for some\ $\varepsilon>0$. Moreover,\ $\xi$ is unique up to replacement by\ $\lambda \xi$ where\ $\vert \lambda \vert=1$.\\

\end{lem}

\begin{proof}
(1) Lemma \ref{lem:classification} and Lemma \ref{lem:devloping map} imply that there exists a locally schlicht function $f:\mathbb{H}\rightarrow \mathbb{H}$ such that $(e^{iz})^{*}\mathrm{d} \sigma^2=f^{*}\mathrm{d} \sigma^2_{\mathbb{H}}$ and $f(z+2\pi)=f(z)+t$ for all $z\in \mathbb{H}$, where  $t\neq0$ is a real number.

(i) If $t<0$, then
$$\widetilde{f}=\begin {pmatrix}
\sqrt{\frac{2\pi}{-t}} & 0 \\ 0 & \sqrt{\frac{-t}{2\pi}} \end {pmatrix}\circ f=-\frac{2\pi}{t}f$$
is also a developing map by Lemma \ref{lem:devloping map}. We have $\widetilde{f}(z+2\pi)=\widetilde{f}(z)-2\pi$. Let $g(z)=\widetilde{f}(z)+z$, then $g(z+2\pi)=g(z)$. And $g(z)$ is a simply periodic function with period $2\pi$. Let $w=e^{iz}$, then there exists a unique holomorphic function $G$ in $\mathbb{D}^{*}=\{w|0<|w|<1\}$ such that $g(z)=G(w)$. Thus by \rm{\cite[p. 264]{A79}} we have the complex Fourier development
\[
g(z)=\sum_{n=-\infty}^{\infty} a_{n}e^{niz}.
\]
So $\widetilde{f}(z)=\sum_{n=-\infty}^{\infty} a_{n}e^{niz}-z$, and $\textup{Im} \widetilde{f}>0$, i.e.
\begin{equation}
\label{equ:(1)}
\textup{Im}\sum_{n=-\infty}^{\infty} a_{n}w^n-\ln\frac{1}{|w|}>0 \ \ \ for\ 0<|w|<1.
\end{equation}
We know that $G(w)=\sum_{n=-\infty}^{\infty} a_{n}w^n$ is a holomorphic function in $\mathbb{D}^{*}$.

If $0$ is a removable singularity of $G(w)$, it contradicts the above inequality as $w$ tends to $0$.

Suppose that $0$ is a pole of order $m$ for $G(w)$. Then by (\ref{equ:(1)}) we have $\textup{Im} \sum_{n=-\infty}^{\infty} a_{n}w^n=\textup{Im}\ w^{-m}h(w)>0$ for $0<|w|<1$, where $h(w)$ is a holomorphic function in $\mathbb{D}$. Contradiction!

Suppose that $0$ is an essential singularity of $G(w)$. Then it contradicts (\ref{equ:(1)}) by the Great Picard Theorem.

Therefore, we have excluded case $t<0$.

(ii) If $t>0$, then $$\widetilde{f}=\begin {pmatrix}
\sqrt{\frac{2\pi}{t}} & 0 \\ 0 & \sqrt{\frac{t}{2\pi}} \end {pmatrix}\circ f=\frac{2\pi}{t}f$$
is also a developing map by Lemma \ref{lem:devloping map}. We have $\widetilde{f}(z+2\pi)=\widetilde{f}(z)+2\pi$.
Let $g(z)=\widetilde{f}(z)-z$, then $g(z+2\pi)=g(z)$. And $g(z)$ is a simply periodic function with period $2\pi$. Let $w=e^{iz}$, then there exists a unique holomorphic function $G$ in $\mathbb{D}^{*}=\{w|0<|w|<1\}$ such that $g(z)=G(w)$. Thus we have the complex Fourier development
\[
g(z)=\sum_{n=-\infty}^{\infty} a_{n}e^{niz}.
\]
So $\widetilde{f}(z)=\sum_{n=-\infty}^{\infty}a_{n}e^{niz}+z$, and $\textup{Im} \widetilde{f}>0$, i.e.
\begin{equation}
\label{equ:(2)}
\textup{Im}\sum_{n=-\infty}^{\infty} a_{n}w^n+\ln\frac{1}{|w|}>0 \ \ \ for\ 0<|w|<1.
\end{equation}
We know that $G(w)=\sum_{n=-\infty}^{\infty} a_{n}w^n$ is a holomorphic function on $\mathbb{D}^{*}$. By (\ref{equ:(2)}) we have
\[
e^{\textup{Im} G(w)}=\bigg{|}e^{-\sqrt{-1} G(w)}\bigg{|}>|w| \ \ \ for\ 0<|w|<1.
\]
So
$\bigg{|}\frac{e^{-\sqrt{-1} G(w)}}{w}\bigg{|}>1$ and $\frac{e^{-\sqrt{-1} G(w)}}{w}$ is a holomorphic function in $\mathbb{D}^{*}$. By the Great Picard Theorem, $0$ is not an essential singularity of $\frac{e^{-\sqrt{-1} G(w)}}{w}$.

If $G(w)$ has an essential singularity at $0$, consider any non-zero $c\in\mathbb{C}$. By the Casorati-Weierstrass theorem \rm{\cite[p.86]{S2003}}, there is a sequence $z_{n}\rightarrow0$ such that $G(z_{n})\rightarrow \sqrt{-1}\log c$. So $\exp(-\sqrt{-1}G(z_{n}))\rightarrow c$. Since this is true for all non-zero $c$, $\exp(-\sqrt{-1}G(w))$ must have an essential singularity at $0$, then $\frac{e^{-\sqrt{-1} G(w)}}{w}$ does too. Contradiction!

Suppose that $0$ is a pole of order $m$ for $G(w)$. Then
\[
-\sqrt{-1}G(w)=\frac{h(w)}{w^{m}},
\]
where $h(w)$ is holomorhic on $\mathbb{D}$ and does not vanish near the origin. Let $h(0)=re^{\sqrt{-1}\theta}$, $r>0$. Consider the sequence
\[
z_{k}=\frac{\exp(\sqrt{-1}\theta/m)}{k},
\]
then $-\sqrt{-1}G(z_{k})=h(z_{k})\exp(-\sqrt{-1}\theta)k^{m}$. Since $h(z_{k})\rightarrow re^{\sqrt{-1}\theta}$, $\exp(-\sqrt{-1}G(z_{k}))$ converges to $+\infty$. If we consider
\[
w_{k}=\frac{\exp(\sqrt{-1}(\pi+\theta)/m)}{k},
\]
then we have $\exp(-\sqrt{-1}G(w_{k}))$ converges to $0$. So $\exp(-\sqrt{-1}G(w))$ have an essential singularity at $0$, contradiction!

Hence $G(w)$ extends to $w=0$ holomorphically.

 So $\widetilde{f}(z)=\sum_{n=k}^{\infty} a_{n}e^{niz}+z$, where $k\geq0$. And $\widetilde{f}(w)=-\sqrt{-1}\log w+\sum_{n=k}^{\infty} a_{n}w^{n}$ can be viewed as a developing map from $\mathbb{D}^{*}$ to $\mathbb{H}$. So we can choose another complex coordinate $\xi$ near $0$ with\ $\xi(0)=0$ such that
$$\xi=w\cdot\exp\bigg(\sqrt{-1}\sum_{n=k}^{\infty} a_{n}w^{n}\bigg),$$
then
$$\mathrm{d} \sigma^2|_{\Delta_{\varepsilon}}=\widetilde{f}^{*}\mathrm{d} \sigma^2_{\mathbb{H}}=\vert \xi \vert ^{-2}\big(\ln|\xi|\big)^{-2}|d\xi|^{2},$$
where\ $\Delta_{\varepsilon}=\{{w\in\mathbb{C}|\vert w \vert<\varepsilon}\}$\ for some\ $\varepsilon>0$.

Here we show the uniqueness of the complex coordinate $\xi$. Let\ $\xi$ and\ $\widetilde{\xi}$\ be coordinates such that conditions of the lemma are satisfied, then\ $F(\xi)=-\sqrt{-1}\log \xi$,\ $\widetilde{F}(\widetilde{\xi})=-\sqrt{-1}\log \widetilde{\xi}$ are all developing maps of\ $\mathrm{d} \sigma^2$. By Remark \ref{rem:devloping map}, there exists\ $\mathfrak{L}\in \textup{PSL}(2,\mathbb{R})$ such that\ $\widetilde{F}=\mathfrak{L}\circ F$, then\ $\widetilde{F}=\frac{aF+b}{cF+d}\ ,\ ad-bc=1$. Since\ $\xi(0)=\widetilde{\xi}(0)=0$,\ $F(0)=\widetilde{F}(0)=\infty$, we have\ $c=0$  by a calculation. Thus\ $\widetilde{F}=\frac{aF+b}{d}=a^{2}F+ab$, then\ $-\sqrt{-1}\log\widetilde{\xi}=-a^{2}\sqrt{-1}\log \xi+ab$, $\widetilde{\xi}=\xi^{a^{2}}\cdot e^{\sqrt{-1}ab}$. So there exists an open disk\ $V$\ which is near $0$ and does not contain\ $0$\ such that\ $a^{2}=1$,\ $\log\widetilde{\xi}=\log \xi+ab\sqrt{-1}$. Therefore we have\ $\widetilde{\xi}=\lambda \xi$ on\ $V$ with\ $|\lambda|=1$. Since\ $\xi$,\ $\widetilde{\xi}$ and\ $w$ are coordinates near\ $0$,\ $z$ and\ $\widetilde{\xi}$\ are holomorphic functions of\ $w$, then\ $\widetilde{\xi}=\lambda \xi$,\ $|\lambda|=1$\ holds in a neighborhood of\ $0$.\\

(2) As in case (1) there exists a holomorphic function $f:\mathbb{H}\rightarrow\mathbb{D}$ such that $(e^{iz})^{*}\mathrm{d} \sigma^2=f^{*}\mathrm{d} \sigma^2_{\mathbb{D}}$ and that $f(z+2\pi)=e^{2\pi\alpha i}f(z),\ 0<\alpha<1$. Let $g(z)=f\cdot\exp {(-i\alpha z)}$, then $g(z+2\pi)=g(z)$. And $g(z)$ is a simply periodic function with period $2\pi$. Let $w=e^{iz}$, then there exists a unique holomorphic function $G$ in $\mathbb{D}^{*}=\{w|0<|w|<1\}$ such that $g(z)=G(w)$. Thus we have the complex Fourier development
\[
g(z)=\sum_{n=-\infty}^{\infty} a_{n}e^{niz}.
\]
We know that $G(w)=\sum_{n=-\infty}^{\infty} a_{n}w^n$ is a holomorphic function in $\mathbb{D}^{*}$. We have $|G(w)|\cdot|w|^{\alpha}<1$ by the range of $f$, then we have $w=0$ is a removable singularity of $G(w)$ by $|G(w)|<|w|^{-\alpha}$ and $0<\alpha<1$. So
\[
f(w)=w^{\alpha}\sum\limits_{n=k}^\infty a_{n}w^{n},
\]
where $k(\geq0)$ is an integer and $a_{k}\neq0$. And $f(w)$ can be viewed as a developing map from $\mathbb{D}^{*}$ to $\mathbb{D}$. So we can choose another complex coordinate $\xi$ near $0$ with\ $\xi(0)=0$ such that $\xi^{\alpha+k}=w^{\alpha}\sum\limits_{n=k}^\infty a_{n}w^{n}$,
then
$$ \mathrm{d} \sigma^2|_{\Delta_{\varepsilon}}=\frac{4(k+\alpha)^{2}\vert \xi \vert^{2k+2\alpha-2}}{(1-\vert \xi \vert ^{2k+2\alpha})^{2}}{\vert \mathrm{d} \xi \vert}^{2},$$
where\ $\Delta_{\varepsilon}=\{{w\in\mathbb{C}|\vert w \vert<\varepsilon}\}$\ for some\ $\varepsilon>0$.

We show the uniqueness of the complex coordintae $\xi$. Let\ $\xi$ and\ $\widetilde{\xi}$\ be coordinates such that conditions of the lemma are satisfied, then\ $F(\xi)=\xi^{\alpha}$,\ $\widetilde{F}(\widetilde{\xi})=\widetilde{\xi}^{\alpha}$\ are all developing maps of\ $ \mathrm{d} \sigma^2$. By Lemma \ref{lem:devloping map}, there exists\ $\mathfrak{L}\in \textup{PSU(1,1)}$ such that\ $\widetilde{F}=\mathfrak{L}\circ F$, then\ $\widetilde{F}=\frac{aF+b}{\overline{b}F+\overline{a}}\ ,\ \ |a|^{2}-|b|^{2}=1$. Since\ $\xi(0)=\widetilde{\xi}(0)=0$, $F(0)=\widetilde{F}(0)=0$, we have\ $b=0$  by a calculation. Thus\ $\widetilde{F}=\frac{a}{\overline{a}}F=\mu F ,\ |\mu|=1$, then there exists an open disk\ $V$\ which is near $0$ and does not contain\ $0$ such that\ $\widetilde{\xi}^{\alpha}=\mu \xi^{\alpha}$. Therefore we have\ $\widetilde{\xi}=\lambda \xi$ on\ $V$ with\ $|\lambda|=1$. Since\ $\xi$,\ $\widetilde{\xi}$ and\ $w$ are coordinates near\ $0$,\ $\xi$ and\ $\widetilde{\xi}$\ are holomorphic functions of\ $w$, then\ $\widetilde{\xi}=\lambda \xi$,\ $|\lambda|=1$\ holds in a neighborhood of\ $0$.\\

(3) Since $\mathfrak{L}$ is the identity, $f(z+2\pi)=f(z)$, then $f(z)$ is a simply periodic function with period $2\pi$. Let $w=e^{iz}$, then there exists a unique holomorphic function $F$ in $\mathbb{D}^{*}=\{w|0<|w|<1\}$ such that $f(z)=F(w)$. Thus we have the complex Fourier development
\[
f(z)=\sum_{n=-\infty}^{\infty} a_{n}e^{niz}.
\]
We know that $F(w)=\sum_{n=-\infty}^{\infty} a_{n}w^n$ is a holomorphic function in $\mathbb{D}^{*}$ and $|F|=|f|<1$, so $w=0$ is a removable singularity and $F(w)$ extends to $w=0$ holomorphically. Let
\[
F(w)=\sum\limits_{n=k}^\infty a_{n}w^{n},
\]
where $k(\geq0)$ is an integer and $a_{k}\neq0$. And $F(w)$ can be viewed as a developing map from $\mathbb{D}^{*}$ to $\mathbb{D}$. Since $\frac{aF+b}{\overline{b}F+\overline{a}}$ is also a developing, where $a, b\in\mathbb{C}\ \text{and}\ \vert a \vert^2-\vert b \vert^2=1$, so we can set $F(0)=0$ without loss of generality. Then we can choose another complex coordinate $\xi$ near $0$ with\ $\xi(0)=0$ such that $\xi^{k}=\sum\limits_{n=k}^\infty a_{n}w^{n}$,
then
$$ \mathrm{d} \sigma^2|_{\Delta_{\varepsilon}}=\frac{4k^{2}\vert \xi \vert^{2k-2}}{(1-\vert \xi \vert ^{2k})^{2}}{\vert \mathrm{d} \xi \vert}^{2},$$
where $k$ is a positive integer, $\Delta_{\varepsilon}=\{{w\in\mathbb{C}|\vert w \vert<\varepsilon}\}$\ for some\ $\varepsilon>0$.

The uniqueness of the coordinate $\xi$ is similar to the above.

\end{proof}

So we have obtained the Theorem \ref{thm:main}. Note that in the proof of the above lemma, we actually obtain the local expressions of $\mathrm{d} \sigma^2$ near the origin by choosing a special developing map under a suitable complex coordinate. From the proof of the above lemma and
Lemma \ref{lem:devloping map}, we can get the Theorem \ref{thm:local model of developing map}.

% It seems that some results by Lisa Goldberg are useful for me to study
% the moduli of trivial reducible metrics on the two-sphere.

% The background from HCMU metrics should be included in Section 4.

% I might obtain some new stimulation after learning the Theta functions and
%the Theta divisors on compact Riemann surfaces

\end{document}